\documentclass[11pt]{amsart}

\usepackage{accents}
\usepackage{amsmath}
\usepackage{amsfonts}
\usepackage{amssymb}
\usepackage{amsthm}
\usepackage{array,booktabs,multirow}
\usepackage{cite}
\usepackage{dsfont}
\usepackage[shortlabels]{enumitem}
\usepackage{etoolbox}
\usepackage[hang, flushmargin]{footmisc}
\usepackage{float}
\usepackage{hyperref}
\usepackage{latexsym}
\usepackage{needspace}
\usepackage{tikz}
\usetikzlibrary{matrix,arrows}

\theoremstyle{plain}
\newtheorem{thm}{Theorem}[section]
\newtheorem{cor}[thm]{Corollary}
\newtheorem{lem}[thm]{Lemma}
\newtheorem{prop}[thm]{Proposition}

\theoremstyle{definition}

\newtheorem{defn}[thm]{Definition}
\newtheorem{exam}[thm]{Example}

\AtBeginEnvironment{lem}{\Needspace{4\baselineskip}}
\AtBeginEnvironment{thm}{\Needspace{6\baselineskip}}
\AtBeginEnvironment{prop}{\Needspace{6\baselineskip}}
\AtBeginEnvironment{cor}{\Needspace{4\baselineskip}}
\AtBeginEnvironment{algo}{\Needspace{6\baselineskip}}

\newcommand{\beast}{\begin{eqnarray*}}
\newcommand{\eeast}{\end{eqnarray*}}

\newcommand{\C}{\mathbf C}
\newcommand{\D}{\mathbf D}

\newcommand{\R}{\mathcal R}
\newcommand{\dom}{{\rm dom}\,}
\newcommand{\cod}{{\rm cod}\,}
\newcommand{\Aut}{{\rm Aut}}
\newcommand{\Hom}{{\rm Hom}}

\title[A unified approach to the Galois closure problem]{A unified approach to the Galois closure problem}

\author{Hau-Wen Huang}
\address{
Department of Mathematics\\
National Central University\\
Chung-Li 32001 Taiwan
}
\email{hauwenh@math.ncu.edu.tw}
\thanks{This research started when both authors were working at the National Center for Theoretical Sciences in Hsinchu, Taiwan. The research of the first author is partially supported by the National Center for Theoretical Sciences in Taiwan, the Council for Higher Education of Israel and the Ministry of Science and Technology of Taiwan under the project MOST 105-2115-M-008-013.}

\author{Wen-Ching Winnie Li}
\address{Department of Mathematics, Pennsylvania State University,
University Park, PA 16802, USA} \email{wli@math.psu.edu}

\thanks{
The research of the second author is partially supported by the NSF grant DMS-1101368 and the Simons Foundation grant \# 355798. Part of the research was done when the author was visiting the National Tsing Hua University in summer 2015 and the Institute of Mathematics, Academia Sinica in spring and summer 2016. She would like to thank NCTS, the Mathematics Department of NTHU, and the Institute of Mathematics, Academia Sinica for their hospitality.}

\begin{document}

\begin{abstract}
In this paper we give a unified approach in categorical setting to the problem of finding the Galois closure of a finite cover, which includes as special cases the familiar finite separable field  extensions, finite unramified covers of a connected undirected graph, finite covering spaces of a locally connected topological space, finite \'etale covers of a smooth projective irreducible algebraic variety, and finite covers of normal varieties. We present two algorithms whose outputs are shown to be desired Galois closures. An upper bound of the degree of the Galois closure under each algorithm is also obtained.
\end{abstract}

\maketitle



\section{Introduction}
The Galois closure of a separable field extension  $F/E$ is a minimal Galois extension over $E$ containing $F$. It is unique up to isomorphism over $E$.  When $F=E(\alpha)$ is a finite simple extension, its Galois closure is the splitting field of the minimal polynomial $f(x)$ of $\alpha$ over $E$. The most common approach to construct the splitting field $K$ of $f(x)$ is the following iterative method: Initially set $K=F$ and $n=1$. Begin with (1).

\setlist[enumerate,1]{leftmargin=2.7em}
\begin{enumerate}
\item[(1)] If $n\geq \deg f(x)-1$, then $K/E$ is the Galois closure of $F/E$. Else go to (2).

\item[(2)] Find all monic irreducible factors $p_1(x), p_2(x), \ldots, p_m(x)$ of $f(x)$ in $K[x]$  and go to (3).

\item[(3)] If $\deg p_i(x)=1$ for all $1\leq i\leq m$, then $K/E$ is the Galois closure of $F/E$. Else go to (4).

\item[(4)] Choose an irreducible factor $p_i(x)$ with $\deg p_i(x)>1$ and go to (5).

\item[(5)] Reset $K$ to be $K[x]/(p_i(x))$ and increase $n$ by one. Go back to (1).
\end{enumerate}

Given a finite cover, such as an unramified cover of a connected undirected graph, or a finite covering space of a locally connected topological space, or a finite \'etale cover of a smooth projective irreducible algebraic variety, it is natural to ask the similar question of finding an explicit algorithm which outputs the Galois closure of a given cover.
In this paper, we give a unified approach to the Galois closure problem, including the aforementioned covers, by employing the language of category theory to formulate conditions (G1)--(G4) under which an iterative algorithm,  Algorithm $\mathds I$, is shown to output a Galois closure.  The same holds for another divide-and-conquer algorithm, Algorithm $\mathds R$,
under the additional hypothesis (G5). The conditions (G1)--(G5) are satisfied by the familiar covers discussed above.

This paper is organized as follows. The conditions (G1)--(G4) in categorical setting which encompass all special cases of interest are formulated in \S\ref{sec:category} and the condition (G5) is stated in \S\ref{sec:algo}. In \S\ref{sec:galois&homset} we study covers and characterize a sufficient condition for Galois closures.
The Algorithms $\mathds I$ and $\mathds R$ are introduced in \S\ref{sec:algo}, where we show that the outputs of these two algorithms are indeed Galois closures of the input cover $f$. The Galois closures of $f$ are shown to be unique up to isomorphism. As consequences, the degree of the Galois closures of $f$ is less than or equal to $(\deg f)!$; if the additional condition (G5) holds,  then the degree of the Galois closures of $f$ is shown to be a factor of $(\deg f)!$. In arithmetic geometry, finite covers between irreducible varieties defined over a field play a very important role. In particular, they are closely tied to the celebrated Inverse Galois Problem, which concerns realizing a finite group as the Galois group of a prescribed base field. Thus it would be desirable to have a categorical approach to finding the Galois closure of a given finite cover allowing ramification. We address  this question in the final section, where it is explained that the two algorithms in \S\ref{sec:algo} can be used to obtain Galois closures of finite covers, with or without ramifications, of normal varieties.

\section{A unified categorical setting}\label{sec:category}

In this paper we assume familiarity with basic categorical notions and terminologies such as categories, objects, arrows, epis, monos, pushouts, and pullbacks. The reader is referred to \cite[Chapters 1 and 3]{Lane:98} for more detail. Given two arrows $f$, $g$ in a category with the same codomain, namely $\cod f = \cod g$, denote by
$$
\Hom(g,f)
$$
the collection of all arrows $h$ such that $g=f\circ h$.
For notational
convenience, we adopt the above notation instead of the hom-sets in a comma category as in \cite[Chapter 2, \S6]{Lane:98}.

Throughout this paper we consider a category $\C$ which contains a full subcategory $\D$ satisfying the properties (G1)--(G4) below.

\setlist[enumerate,1]{leftmargin=2.7em}
\begin{enumerate}

\item[{\rm (G1)}] Any diagram $B\rightarrow A\leftarrow C$ in $\D$ has a pullback in $\C$.

\item[{\rm (G2)}] The category $\D$ satisfies the following properties:
\begin{enumerate}
\item[{\rm (I)}] The pushouts exist.

\item[{\rm (II)}] Each arrow is epic.

\item[{\rm (III)}] If an arrow is monic, then it is an isomorphism, that is, its inverse is an arrow of $\D$.
\end{enumerate}
\end{enumerate}

\noindent

\begin{enumerate}
\item[{\rm (G3)}]  For each object $U$ of $\C$, there exists a set $\Sigma(U)$ of arrows $i$ of $\C$ with $\dom i$ in $\D$ and $\cod i=U$ satisfying the following property: For any arrow $u$ with $\dom u$ in $\D$ and $\cod u=U$ there exists a unique $i\in \Sigma(U)$ with $\Hom(u,i)\not=\emptyset$; in other words, any such $u$ factors through a unique $i$ in $\Sigma(U)$.
\end{enumerate}

\begin{enumerate}
\item[{\rm (G4)}]
There is a degree function``$\deg$'' from the collection $\R$ of arrows of $\C$ with codomains in $\D$ into the set of positive integers with the following properties:

\begin{enumerate}
\item[{\rm (I)}] $\deg (g\circ f)=\deg g\cdot \deg f$
for any arrows $f$, $g$ and $g\circ f$ in $\R$.

\item[{\rm (II)}]
$
\deg f=\sum_{i\in \Sigma(\dom\! f)} \deg (f\circ i)
$
 for any arrow $f$ in $\R$.

\item[{\rm (III)}]
If $B\stackrel{f}{\rightarrow}A\stackrel{g}{\leftarrow} C$ is a diagram in $\D$, then its pullback $B\stackrel{p}{\leftarrow} U\stackrel{q}{\rightarrow} C$ satisfies that $\deg f=\deg q$ and $\deg g=\deg p$.

\end{enumerate}
\end{enumerate}

\medskip

This categorical setting encompasses the following examples of special interests to us.

\begin{exam}\label{ex:field} (Commutative separable algebras over a field).
Fix a field $E$. Let $\C^*$ be the category whose objects are the nonzero finite-dimensional commutative separable $E$-algebras and whose arrows are nonzero algebra homomorphisms over $E$. Set $\D^*$ to be the full subcategory of $\C^*$ generated by the objects which are fields. Then an arrow $f: F \to K$ in $\D^*$ is an embedding over $E$ of the field $F$ into the field $K$.

Without loss of generality we may regard the arrows of $\D^*$ as finite separable field extensions.
For any two field extensions $K/F$ and $L/F$, the tensor product $K\otimes_F L$ is a pushout of $K/F$ and $L/F$. For any two extensions $L/F$ and $L/K$, the field $F\cap K$ is a pullback of $L/F$ and $L/K$.

Each object of $\C^*$ is a finite direct sum of finite separable field extensions of $E$, which are objects of $\D^*$.
For each object $U$ of $\C^*$, let the set $\Sigma(U)$ consist of the projections from $U$ to each of its direct summands. Then an arrow $u$ from $U$ to an object $K$ in $\D^*$ has exactly one nontrivial restriction to its direct summand since $K$ does not contain nonzero zero divisors.
Hence there is a unique arrow $i \in \Sigma(U)$ such that $u = f \circ i$ for some arrow $f$ of $\D^*$.
Set
$$
\deg(F\rightarrow X)=\dim_F X
\qquad
\hbox{for all arrows $F\rightarrow X$ with domain $F$ in $\D^*$}.
$$
It follows from the above discussion that the opposite categories of $\C^*$ and $\D^*$, denoted by $\C$ and $\D$, respectively, satisfy (G1)--(G4).

Alternatively, we can concretely describe the categories $\C$ and $\D$ as follows. For each object $X$ of $\C^*$ its dual $X^* = \Hom_E(X, E)$  is the set of $E$-module homomorphisms from $X$ to $E$. Clearly $X^*$ is a free $E$-module of the same dimension over $E$ as $X$. Each arrow $f: X \to Y$ in $\C^*$ gives rise to an arrow $f^* : Y^* \to X^*$ sending $\phi \in Y^*$ to  $f^*(\phi) := \phi \circ f$ in $X^*$. Call $f^*$ the dual of $f$. Then $\C$ (resp. $\D$) can be regarded as the category whose objects and arrows are the duals of those of $\C^*$ (resp. $\D^*$).
\end{exam}

\begin{exam}\label{ex:graph} (Finite covers of a connected undirected graph).
Fix a connected undirected graph $G$.
Let $\C$ be the category whose objects are finite unramified covers of $G$ and whose arrows are finite unramified covering maps of undirected graphs. Such an arrow is a local graph isomorphism such that the preimage of each vertex in the codomain, called a fiber, consists of finitely many vertices.
Choose $\D$ to be the full subcategory of $\C$ generated by the connected graphs.

 For an undirected graph $X$, denote by $V(X)$ its vertex set and $E(X)$ its edge set. Let $f:Y\to X$ and $g:Z\to X$ be two arrows of $\D$.
Recall from  \cite{Hell:72} that the fiber product $Y\times_X Z$ is an undirected graph whose vertices are the pairs $(b,c)\in V(Y)\times V(Z)$ satisfying $f(b)=g(c)$. An edge connecting two vertices $(b,c)$, $(b',c')$ of $Y\times_X Z$ is
a pair $(x,y)\in E(Y)\times E(Z)$ such that $f(x)=g(y)$, $x$ connects $b$ and $b'$, and $y$ connects  $c$ and $c'$. The graph $Y\times_X Z$ together with the canonical projections $Y\times_X Z\to Y$ and $Y\times_X Z\to Z$ is a pullback of $Y\stackrel{f}{\rightarrow}X\stackrel{g}{\leftarrow} Z$. Therefore (G1) holds.

Given two arrows $f:X\to Y$ and $g:X\to Z$ of $\D$, let $Y \amalg_X Z$ be the connected graph obtained from a disjoint union of $Y$ and $Z$ by identifying the vertices and edges which have the same preimages in $X$ under $f$ and $g$. Then $Y \amalg_X Z$ along with the canonical maps $Y\to Y \amalg_X Z$ and $Z\to Y \amalg_X Z$ is a pushout of $Y\stackrel{f}{\leftarrow} X\stackrel{g}{\rightarrow} Z$. Therefore (G2-I) holds. (G2-II) and (G2-III) are obvious.

To see (G3), note that each object $U$ of $\C$ is a disjoint union of finitely many connected components $X_i$ of $U$. Let $\Sigma(U)$ consist of the inclusions from $X_i$ to $U$. Given any arrow $u : Y \to U$ with $Y$ in $\D$, since $Y$ is connected, the image of $u$ is a connected component of $U$, say, $X_i$. Hence $u$ factors through the unique inclusion map from $X_i$ to $U$. Finally, we define the degree of an arrow $X\rightarrow Y$ of $\C$. Since $Y$ is connected, the cardinalities of the fibers over the vertices of $Y$ are the same, this number is $\deg (X\rightarrow Y)$. In other words,
$$
\deg (X\rightarrow Y)= | {\rm fib}_Y X |.
\qquad
\hbox{for all arrows $X\rightarrow Y$ with codomain $Y$ in $\D$}.
$$
 Then $\deg$ satisfies (G4).
\end{exam}

\begin{exam}\label{ex:topology} (Finite covers of a connected topological space).
This is the same as the previous example, with graphs replaced by locally connected topological covering spaces of a connected topological space $M$.
\end{exam}

\begin{exam}\label{ex:variety} (Finite \'etale covers of an irreducible algebraic variety).
Let $V$ be an irreducible algebraic variety over a field $k$.
Set $\C$ to be the category whose objects are finite \'etale covers of $V$ and whose arrows are \'etale covering morphisms, and let $\D$ be the full subcategory generated by irreducible \'etale covers. This can be seen in two ways: geometrically as covering spaces similar to Examples~\ref{ex:graph} and \ref{ex:topology}, or algebraically as the dual category of the category of nonzero finite-dimensional commutative separable algebras over the field $E$ of $k$-rational functions on $V$ as in Example~\ref{ex:field}. However, it differs from the previous three examples in that the universal cover for the objects in the category $\D$ for the first three examples is again an object of the same nature, namely the algebraic closure of the field $E$ in Example~\ref{ex:field}, the maximal unramified connected graph cover of $G$ in Example~\ref{ex:graph}, and the maximal locally connected topological covering space of $M$ in Example~\ref{ex:topology}, but the universal object for finite irreducible \'etale covers of $V$ is no longer a variety.
\end{exam}

\section{Hom-sets and Galois closures}\label{sec:galois&homset}

Motivated by Examples \ref{ex:field}--\ref{ex:variety},
an arrow of $\D$ is called a {\it cover}\,; it has finite degree by (G4).  In this section we study properties of covers. In \S\ref{s:degone} we characterize the isomorphisms in $\D$ as covers of degree one.
In \S\ref{sec:hompullback} we describe $\Hom(g,f)$ for covers $f$ and $g$ with the same codomain in terms of their pullback. As a consequence we see that the size of $\Hom(g, f)$ is at most $\deg f$.
For any arrow $f$ let $\Aut(f)$ be the collection of all isomorphisms $g\in \Hom(f,f)$. Therefore the cardinality of $\Aut(f)$ is at most $\deg f$ for any cover $f$.

We define a Galois cover analogous to that in field extensions.

\begin{defn}\label{defn:galois}
A cover $f$ is called {\it Galois} {if $\Aut(f)$ contains exactly $\deg f$ covers.}
\end{defn}

Define a preorder $\sqsubseteq$ on the collection of arrows of $\C$ by $f\sqsubseteq g$ if $\cod f=\cod g$ and $\Hom(g, f)$ is nonempty.
A cover $g$ is said to be {\it least} with respect to the property $P$ if $g$ has property $P$ and $g\sqsubseteq h$ for any cover $h$ with $P$.
Call a cover $g$ {\it minimal} with respect to the property $P$ if  (1) $g$ has property $P$ and (2) $g\sqsubseteq h$ for any cover $h$ with $P$ and satisfying $h\sqsubseteq g$. Thus a least cover is necessarily minimal.

\begin{defn}\label{defn:galoisclosure}
A cover $g$ is called a {\it Galois closure} of a cover $f$ if $g$ is a minimal cover with respect to the property that $g$ is Galois and $f\sqsubseteq g$.
\end{defn}

In \S\ref{sec:galoisclosure}, for any cover $f$ we give a sufficient condition for a cover $g$ with $f\sqsubseteq g$ to be a Galois closure of $f$.

In general, two covers $f : A \to B$ and $g : C \to D$ are said to be isomorphic if there exist isomorphisms $p:A \to C$
and $q:B \to D$ such that $q\circ f=g\circ p$. In this paper we only consider the special case where $B=D$ by defining
 two covers $f$ and $g$ to be {\it isomorphic} if
$f \sqsubseteq g$ and $g \sqsubseteq f$.  As we shall see from Theorem \ref{thm:existgc} that the Galois closures of $f$ exist, and are unique up to isomorphism.
Furthermore every Galois closure $g$ of $f$ is least with respect to the property that $g$ is Galois and $f \sqsubseteq g$.

\subsection{The degree one covers as isomorphisms}\label{s:degone}
 In this subsection, we show that the isomorphisms in $\D$ are exactly the covers with degree one. We begin by observing some simple consequences of the condition (G4).

\begin{lem}\label{lem:deg<=}
If $f\sqsubseteq g$ for two covers $f$ and $g$, then $\deg f$ divides $\deg g$.
\end{lem}
\begin{proof}
Pick any $h\in \Hom(g,f)$. By (G4-I) we have $\deg g=\deg f\cdot \deg h$. Since $\deg h$ is a positive integer, the lemma follows.
\end{proof}

\begin{lem}\label{lem:1A}
$\deg 1_A=1$ for all objects $A$ of $\D$.
\end{lem}
\begin{proof}
By (G4-I) we have
$$
(\deg 1_A)^2=\deg (1_A\circ 1_A)=\deg 1_A.
$$
Since $\deg 1_A$ is positive, it follows that $\deg 1_A=1$.
\end{proof}

Now we prove the main result of this subsection.

\begin{prop}\label{prop:iso}
A cover $f:B\to A$ is an isomorphism if and only if $\deg f=1$.
\end{prop}
\begin{proof}
($\Rightarrow$):
Since $f^{-1}\in \Hom(1_A,f)$ this shows that $f\sqsubseteq 1_A$. By Lemma~\ref{lem:deg<=} we have $\deg f\leq \deg 1_A$ and $\deg 1_A=1$ by Lemma~\ref{lem:1A}. Since $\deg f$ is a positive integer by (G4) this forces that $\deg f=1$.

($\Leftarrow$): By (G1) there exists a pullback $B\stackrel{p}{\leftarrow}U\stackrel{q}{\rightarrow} B$ of $B\stackrel{f}{\rightarrow}A\stackrel{f}{\leftarrow} B$. By (G4-III) we have $\deg p=1$.
Condition (G4-II) implies that
$$
\sum_{i\in \Sigma(U)} \deg (p\circ i)=1.
$$
Therefore $\Sigma(U)$ contains exactly one arrow $i:I\to U$ with $I$ an object of $\D$.

Let $r=p\circ i$ and $s=q\circ i$. By the universal property of  $B\stackrel{p}{\leftarrow}U\stackrel{q}{\rightarrow} B$ the axiom (G3) implies that
$$
B
\stackrel{r}{\longleftarrow}
I
\stackrel{s}{\longrightarrow}
B
$$
is a {\it weak pullback} of $B\stackrel{f}{\rightarrow}A\stackrel{f}{\leftarrow} B$ in $\D$. This means that
$f\circ r=f\circ s$ and if there are two covers $g:C\to B$ and $h:C\to B$ with $f\circ g=f\circ h$ then there exists a cover $u:C\to I$ such that the diagram

\begin{table}[H]
\centering
\begin{tikzpicture}[descr/.style={fill=white}]
\matrix(m)[matrix of math nodes,
row sep=2.6em, column sep=2.8em,
text height=1.5ex, text depth=0.25ex]
{
C
&
&\\
&I
&B\\
&B
&A\\
};
\path[->,font=\scriptsize,>=angle 90]
(m-1-1) edge [bend left] node[above] {$h$} (m-2-3)
        edge [bend right] node[left] {$g$} (m-3-2)
(m-2-2) edge node[above] {$s$} (m-2-3)
        edge node[left] {$r$} (m-3-2)
(m-3-2) edge node[below] {$f$} (m-3-3)
(m-2-3) edge node[right] {$f$} (m-3-3); \path[->,dashed,font=\scriptsize,>=angle 90]
(m-1-1) edge node[descr] {$u$} (m-2-2);
\end{tikzpicture}
\end{table}

\noindent commutes. Take $C = B$, $g=1_{B}$ and $h=1_{B}$.
Then we have
\begin{gather*}
r\circ u=1_B.
\end{gather*}
By (G2-II) the cover $u$ is epic. Applying $u$ to the above equality on the left-hand side, it follows that $u\circ r\circ u=u$. Since $u$ is epic, this implies $u\circ r=1_I$.
Therefore $u$ is an isomorphism. This shows that $B\stackrel{1_{B}}{\leftarrow}B\stackrel{1_{B}}{\rightarrow} B$ is a weak pullback of  $B\stackrel{f}{\rightarrow}A\stackrel{f}{\leftarrow} B$ in $\D$ as well.

Now, suppose there are two covers $g$ and $h$ with $f\circ g=f\circ h$. By the weak universal property of $B\stackrel{1_{B}}{\leftarrow} B\stackrel{1_{B}}{\rightarrow}B$ there exists a cover $u$ such that $g=1_{B}\circ u$ and $h=1_{B}\circ u$. Therefore $g=h$. This shows that $f$ is a mono and hence an isomorphism  by (G2-III).
\end{proof}

As an application of Proposition~\ref{prop:iso}, we obtain some necessary and sufficient conditions for two covers to be isomorphic.

\begin{cor}\label{cor:iso_cover}
Let $f$ and $g$ be two covers satisfying $f\sqsubseteq g$. Then the following are equivalent:
\begin{enumerate}
\item $f$ is isomorphic to $g$;

\item $\deg f=\deg g$;

\item $\Hom(g,f)$ contains an isomorphism.
\end{enumerate}
Suppose {\rm (i)--(iii)} hold. Then all covers in $\Hom(g,f)$ are isomorphisms.
\end{cor}

\begin{proof}
(i) $\Rightarrow$ (ii): The conditions $f\sqsubseteq g$ and $g\sqsubseteq f$ imply $\deg f=\deg g$ by  Lemma~\ref{lem:deg<=}.

(ii) $\Rightarrow$ (iii): Since $f\sqsubseteq g$ the set $\Hom(g,f)$ is nonempty.  Since $\deg f=\deg g$, it follows from Proposition~\ref{prop:iso} that each cover in $\Hom(g,f)$ is an isomorphism.

(iii) $\Rightarrow$ (i): Immediate.
\end{proof}

\begin{cor}\label{cor:aut}
$\Aut(f)=\Hom(f,f)$ for any cover $f$.
\end{cor}
\begin{proof}
The corollary is immediate from Corollary~\ref{cor:iso_cover} since Corollary~\ref{cor:iso_cover}(ii) are satisfied when $g=f$.
\end{proof}

\subsection{Hom-sets and pullbacks}\label{sec:hompullback}

The main goal of this subsection is to describe the hom-set of two covers with the same codomain in terms of their pullbacks.

\begin{lem}\label{lem:pullback}
Let $f:B\to A$ and $g:C\to A$ denote two covers.
If $B\stackrel{p}{\leftarrow} U\stackrel{q}{\rightarrow} C$ is a pullback of  $B\stackrel{f}{\rightarrow}A\stackrel{g}{\leftarrow} C$, then
$$
\deg f=\sum_{i\in \Sigma(U)} \deg (q\circ i).
$$
\end{lem}
\begin{proof}
Combine  (G4-II) and (G4-III).
\end{proof}

The following lemma provides a useful tool to distinguish the arrows in the set $\Sigma(U)$ for any object $U$ of $\C$ which occurs in a pullback of covers.

\begin{lem}\label{lem:comp=}
Let $f:B\to A$ and $g:C\to A$ denote two covers. Let $B\stackrel{p}{\leftarrow} U\stackrel{q}{\rightarrow} C$ be a pullback of $B\stackrel{f}{\rightarrow}A\stackrel{g}{\leftarrow}C$. Two arrows $i$, $j\in \Sigma(U)$ are equal if and only if there exists a cover $u : \dom i  \to \dom j$ such that
\begin{gather*}
p\circ j\circ u=p\circ i
\qquad  and  \qquad
q\circ j\circ u=q\circ i.
\end{gather*}
\end{lem}
\begin{proof}
Write $I$ for $\dom i$.

($\Rightarrow$): Take $u$ to be $1_I$.

($\Leftarrow$): Consider the commutative diagram:

\begin{table}[H]
\centering
\begin{tikzpicture}[descr/.style={fill=white}]
\matrix(m)[matrix of math nodes,
row sep=2.6em, column sep=2.8em,
text height=1.5ex, text depth=0.25ex]
{
I
&\\
&U
&C\\
&B
&A\\
};
\path[->,font=\scriptsize,>=angle 90]
(m-1-1) edge [bend right] node[left] {$p\circ i$} (m-3-2)
        edge [bend left] node[above] {$q\circ i$} (m-2-3)
(m-2-2) edge node[above] {$q$} (m-2-3)
        edge node[left] {$p$} (m-3-2)
(m-3-2) edge node[below] {$f$} (m-3-3)
(m-2-3) edge node[right] {$g$} (m-3-3); \path[->,dashed,font=\scriptsize,>=angle 90]
(m-1-1) edge node[descr] {$i$} (m-2-2);
\end{tikzpicture}
\end{table}

\noindent Observe that replacing the arrow $i$ by $j\circ u$, the above diagram still commutes. The universal property of $B\stackrel{p}{\leftarrow}U\stackrel{q}{\rightarrow}C$ implies that $i=j\circ u$ and hence $u \in \Hom(i,j)$. On the other hand, we have $1_I\in \Hom(i,i)$. Therefore the arrow $i$ factors through $i$ and $j$. Condition (G3) forces that  $i=j$.
\end{proof}

The following theorem establishes a bijection between the hom-set of two covers with the same codomain and the isomorphisms arising from their pullback. Note that the bijection $\Phi$ below is well-defined by Proposition~\ref{prop:iso}.

\begin{thm}\label{thm:homset}
Let $f:B\to A$ and $g:C\to A$ be two covers.
Let $B\stackrel{p}{\leftarrow} U\stackrel{q}{\rightarrow} C$ be a pullback of $B\stackrel{f}{\rightarrow}A\stackrel{g}{\leftarrow} C$. Then the map $\Phi$ from $\{i\in \Sigma(U)~|~\deg (q\circ i)=1\}$ to
$\Hom(g,f)$
given by
$$
i\quad \mapsto \quad p\circ i\circ(q\circ i)^{-1}
$$
is a bijection. In particular $|\Hom(g,f)|\leq \deg f$ and the equality holds if and only if
$$
\deg (q\circ i)=1 \qquad {\rm ~for~ all~~} i\in \Sigma(U).$$
\end{thm}
\begin{proof}
(Surjectivity) Given $h\in \Hom(g,f)$, we need to find an $i\in \Sigma(U)$ with $\deg (q\circ i)=1$ such that $\Phi(i)=h$. By the universal property of $B\stackrel{p}{\leftarrow} U\stackrel{q}{\rightarrow} C$, there exists a unique arrow $u:C\to U$ such that the diagram

\begin{table}[H]
\centering
\begin{tikzpicture}[descr/.style={fill=white}]
\matrix(m)[matrix of math nodes,
row sep=2.6em, column sep=2.8em,
text height=1.5ex, text depth=0.25ex]
{
C
&\\
&U
&C\\
&B
&A\\
};
\path[->,font=\scriptsize,>=angle 90]
(m-1-1) edge [bend right] node[left] {$h$} (m-3-2)
        edge [bend left] node[above] {$1_C$} (m-2-3)
(m-2-2) edge node[above] {$q$} (m-2-3)
        edge node[left] {$p$} (m-3-2)
(m-3-2) edge node[below] {$f$} (m-3-3)
(m-2-3) edge node[right] {$g$} (m-3-3); \path[->,dashed,font=\scriptsize,>=angle 90]
(m-1-1) edge node[descr] {$u$} (m-2-2);
\end{tikzpicture}
\end{table}

\noindent commutes. By (G3) there exists an arrow $i\in \Sigma(U)$ with $\Hom(u,i)\not=\emptyset$. Pick any $j \in \Hom(u,i)$. Since $u=i\circ j$ and $1_C=q\circ u$, it follows that $1_C=(q\circ i)\circ j$. Applying the function $\deg$ to both sides, we obtain that $\deg (q\circ i)=1$ using Lemma~\ref{lem:1A} and (G4-I). By Proposition \ref{prop:iso} the arrow $q\circ i$ is an isomorphism with inverse $j$. The commutativity of the above diagram implies that
$$
\Phi(i) = p\circ i\circ(q\circ i)^{-1} = p \circ i \circ j = p \circ u = h,
$$
as desired.

(Injectivity) Suppose that $\Phi$ sends two arrows $i,j\in \Sigma(U)$ with $\deg (q\circ i)=1$ and $\deg (q\circ j)=1$ to the same cover in $\Hom(g,f)$. Namely
$p\circ i\circ(q\circ i)^{-1}=p\circ j\circ (q\circ j)^{-1}$. Observe that the cover $u=(q\circ j)^{-1}\circ q\circ i$ satisfies
$p\circ j\circ u=p\circ i$
and
$q\circ j\circ u=q\circ i$.
Therefore $i=j$ by Lemma~\ref{lem:comp=}.

It follows from the bijectivity of $\Phi$ and Lemma~\ref{lem:pullback} that
$
|\Hom(g,f)|\leq \deg f
$
and the equality holds if and only if $\deg (q\circ i)=1$ for all $i\in \Sigma(U)$.
This proves the second assertion.
\end{proof}

We conclude this subsection with some consequences of Theorem~\ref{thm:homset}.

\begin{cor}\label{cor:atleast1}
Let $f:B\to A$ denote a cover and
$B\stackrel{p}{\leftarrow} U\stackrel{q}{\rightarrow} B$ denote a pullback of $B\stackrel{f}{\rightarrow} A\stackrel{f}{\leftarrow} B$. Then there exists an arrow $i\in \Sigma(U)$ such that $\deg (p\circ i)=1$.
\end{cor}
\begin{proof}
Since $1_B\in \Hom(f,f)$ the lemma is immediate from Theorem \ref{thm:homset} with $g = f$.
\end{proof}

The second one is to extend the estimate of $|\Hom(g,f)|$ from a cover $g$ to any arrow $f$ in $\C$ with the same codomain in $\D$.

\begin{cor}\label{cor:homC}
Let $g:B\to A$ denote a cover and $f:U\to A$ an arrow of $\C$. Then
$$
|\Hom(g,f)|\leq \deg f.
$$
Moreover the equality holds only if $|\Hom(g,f\circ i)|=\deg (f\circ i)$ for each $i\in\Sigma(U)$.
\end{cor}
\begin{proof}
Consider the set
$$
S=\{(i,k)~|~i\in \Sigma(U)~\hbox{and}~k\in \Hom(g,f\circ i)\}.
$$
By Theorem~\ref{thm:homset} we have
$$
|S|=\sum_{i\in \Sigma(U)}|\Hom(g,f\circ i)|
\leq \sum_{i\in \Sigma(U)}\deg (f\circ i).
$$
By (G4-II) the latter sum is equal to $\deg f$. Therefore
$$
|S|\leq \deg f.
$$
On the other hand, there is a map $\phi:S\to \Hom(g,f)$ defined by
$$
(i,k)
\quad
\mapsto
\quad
i\circ k
\qquad \quad
\hbox{for all $(i,k)\in S$}.
$$
Pick any $u\in \Hom(g,f)$. By (G3) there exists a unique arrow $i\in \Sigma(U)$ such that $\Hom(u,i)\not=\emptyset$. Pick any $k\in \Hom(u,i)$. Then $\phi$ sends $(i,k)$ to $u$. Therefore $\phi$ is surjective and hence
$$
|\Hom(g,f)|\leq |S|.
$$
This lemma follows by combining the above inequalities.
\end{proof}

The third consequence  is to bound the size of $\Aut(f)$.

\begin{cor}\label{cor:galois}
Let $f:B\to A$ be a cover. Then $|\Aut(f)|\leq \deg f$. Furthermore let $B\stackrel{p}{\leftarrow}U\stackrel{q}{\rightarrow} B$ be a pullback of $B\stackrel{f}{\rightarrow}A\stackrel{f}{\leftarrow}B$.
 Then the following are equivalent:
\begin{enumerate}
\item $f$ is Galois.

\item $\deg (p\circ i)=1$ for all $i\in \Sigma(U)$.

\item $\deg (q\circ i)=1$ for all $i\in \Sigma(U)$.
\end{enumerate}
\end{cor}
\begin{proof}
Recall from Corollary~\ref{cor:aut} that $\Aut(f)=\Hom(f,f)$. Thus this corollary is immediate by applying
Theorem~\ref{thm:homset} with $g = f$.
\end{proof}

Next we apply the estimates on cardinalities of hom-sets to draw more conclusions on Galois covers.

\begin{cor}\label{cor:char_galois}
For any cover $g$ the following are equivalent:
\begin{enumerate}
\item $g$ is Galois.

\item $|\Hom(f,g)|=\deg g$ for all covers $f$ with $g\sqsubseteq f$.

\item $|\Hom(g,f)|=\deg f$ for all covers $f$ with $f\sqsubseteq g$.
\end{enumerate}
\end{cor}
\begin{proof}
(ii) $\Rightarrow$ (i) or (iii) $\Rightarrow$ (i): Take $f=g$ and apply Corollary~\ref{cor:aut}.

(i) $\Rightarrow$ (ii): Fix a cover $f$ with $g\sqsubseteq f$ and let $h\in \Hom(f,g)$.
By (G2-II) the cover $h$ is epic and hence
$$
k\circ h
\qquad \quad
\hbox{for all $k\in \Aut(g)$}
$$
are distinct. This shows that
$$
|\Hom(f,g)|\geq \deg g.
$$
Combined with Theorem~\ref{thm:homset}, we have $|\Hom(f,g)|=\deg g$, as desired.

(i) $\Rightarrow$ (iii): Fix a cover $f$ with $f\sqsubseteq g$ and let $h\in \Hom(g,f)$. Clearly $\Aut(h)$ is a subgroup of $\Aut(g)$.
If $i$ and $j$ are two distinct right $\Aut(h)$-coset representatives in $\Aut(g)$, then $h \circ i \ne h \circ j$ for otherwise we would have $i \circ j^{-1} \in \Aut(h)$, a contradiction. Therefore the set
$$
S=\{h\circ i~|~ i \in \Aut(h)\backslash \Aut(g)\}
$$
has cardinality
$$
|S|=\frac{|\Aut(g)|}{|\Aut(h)|}.
$$
We have $|\Aut(g)|=\deg g$ since $g$ is Galois, and $|\Aut(h)|\leq \deg h$ by Corollary \ref{cor:galois}.  Hence
$$
|S|\geq \frac{\deg g}{\deg h}=\deg f,
$$
in which the last equality follows from (G4-I). On the other hand, since $f\circ h=g$, we have $S\subseteq \Hom(g,f)$. Combined with Theorem~\ref{thm:homset}, we obtain
$$
|S|\leq |\Hom(g,f)|\leq \deg f.
$$
The estimates of $|S|$ lead us to $|\Hom(g,f)|=\deg f$, as desired.
\end{proof}

As a by-product of the proof of Corollary \ref{cor:char_galois}, we have the following result.

\begin{cor}\label{cor:galoisup}
Let $f$ and $g$ denote two covers with $f\sqsubseteq g$. If $g$ is Galois, then so is each cover $h\in \Hom(g,f)$.
\end{cor}
\begin{proof}
In  the proof (i) $\Rightarrow$ (iii) of Corollary~\ref{cor:char_galois}, the estimates of $|S|$ also imply that $|\Aut(h)|=\deg h$. Therefore $h$ is Galois.
\end{proof}

\subsection{A sufficient condition for Galois closures}\label{sec:galoisclosure}

For any two covers $f$, $g$ with $f\sqsubseteq g$, the goal of this subsection is to show that if $g$ is minimal with respect to the property $|\Hom(g,f)|=\deg f$ then $g$ is a Galois closure of $f$. To prove this, we need some preparation.

Suppose that $f$, $g$ and $h$ are three covers satisfying $f\sqsubseteq g\sqsubseteq h$. Then every arrow $p\in \Hom(h,g)$ induces a map $p^*:\Hom(g,f)\to \Hom(h,f)$ given by
$$
k \quad \mapsto \quad k\circ p
\qquad \quad
\hbox{for all $k\in \Hom(g,f)$}.
$$

\begin{lem}\label{lem:indmap}
Assume that $f$, $g$ and $h$ are three covers with $f\sqsubseteq g\sqsubseteq h$. Then for any $p\in \Hom(h,g)$ the induced map $p^*:\Hom(g,f)\to \Hom(h,f)$ is injective. In particular, if $|\Hom(g,f)|=\deg f$ then $|\Hom(h,f)|=\deg f$.
\end{lem}
\begin{proof}
Let $p$ be a cover such that $h=g\circ p$. Since $p$ is epic by (G2-II), the map $p^*$ is injective. We know $|\Hom(h,f)|\leq \deg f$ by Theorem~\ref{thm:homset}. Hence, if $|\Hom(g,f)|=\deg f$, then $|\Hom(h,f)|=\deg f$.
\end{proof}

\begin{lem}\label{lem:transhom}
Let $g$, $h$, $p$, and $q$ be four covers such that the following diagram commutes:
\begin{table}[H]
\centering
\begin{tikzpicture}
[descr/.style={fill=white}]
\matrix(m)[matrix of math nodes,
row sep=2.6em, column sep=2.8em,
text height=1.5ex, text depth=0.25ex]
{
D
&C
\\
B
&A
\\
};
\path[->,font=\scriptsize,>=angle 90]
(m-1-1) edge node[above] {$q$} (m-1-2)
        edge node[left] {$p$} (m-2-1)
(m-1-2) edge node[right] {$h$} (m-2-2)
(m-2-1) edge node[below] {$g$} (m-2-2);
\end{tikzpicture}
\end{table}

\noindent Assume that there is a cover $f$ satisfying $f\sqsubseteq g$, $f\sqsubseteq h$ and $|\Hom(h,f)|=\deg f$. Then there exists a unique map $\phi:\Hom(g,f) \to \Hom(h,f)$ such that
$$
\phi(k)\circ q=k\circ p
\qquad \quad
\hbox{for all $k\in \Hom(g,f)$}.
$$
Moreover $\phi$ is injective.
\end{lem}
\begin{proof}
{Since $q$ is epic by (G2-II), the desired property of the map $\phi$ implies the uniqueness of $\phi$.}
Denote by $r$ the cover $g\circ p=h\circ q$. We have $f \sqsubseteq g \sqsubseteq r$ and $f \sqsubseteq h \sqsubseteq r$. By Lemma~\ref{lem:indmap} the map $p^*: \Hom(g, f) \to \Hom(r, f)$ is injective and the map $q^* : \Hom(h, f) \to \Hom(r, f)$ is a bijection because of the assumption $|\Hom(h,f)|=\deg f$.
Then the injective map
$$
\phi=(q^*)^{-1}\circ p^*
$$
has the desired property.
\end{proof}

\begin{thm}\label{thm:suffgalois}
Let $f$, $g$ denote two covers with $f\sqsubseteq g$.
If $g$ is a minimal cover with respect to the property
$|\Hom(g,f)|=\deg f$, then $g$ is a Galois closure of $f$.
\end{thm}
\begin{proof}
By Corollary \ref{cor:char_galois} any Galois cover $h$ with $f\sqsubseteq h\sqsubseteq g$ satisfies $|\Hom(h,f)|=\deg f$. Then  $g\sqsubseteq h$ by the minimality of $g$. Thus, to see that $g$ is a Galois closure of $f$, it remains to show that $g$ is a Galois cover.

Write $f:B\to A$ and $g:C\to A$. Let $C\stackrel{p}{\leftarrow}U\stackrel{q}{\rightarrow} C$ denote a pullback of $C\stackrel{g}{\rightarrow}A\stackrel{g}{\leftarrow}C$.
To prove that $g$ is Galois, by Corollary \ref{cor:galois}, it suffices to show that $\deg (p\circ i)=1$ for all $i\in \Sigma(U)$.

Let $i:I\to U$ denote an arrow in $\Sigma(U)$. Consider the commutative diagram:
\begin{table}[H]
\centering
\begin{tikzpicture}
[descr/.style={fill=white}]
\matrix(m)[matrix of math nodes,
row sep=2.6em, column sep=2.8em,
text height=1.5ex, text depth=0.25ex]
{
I
&C
\\
C
&A
\\
};
\path[->,font=\scriptsize,>=angle 90]
(m-1-1) edge node[above] {$q\circ i$} (m-1-2)
        edge node[left] {$p\circ i$} (m-2-1)
(m-1-2) edge node[right] {$g$} (m-2-2)
(m-2-1) edge node[below] {$g$} (m-2-2);
\end{tikzpicture}
\end{table}
\noindent Since $|\Hom(g,f)|=\deg f$, by applying Lemma~\ref{lem:transhom} we get a bijection $\phi:\Hom(g,f)\to \Hom(g,f)$ such that
$$
\phi(k)\circ q\circ i=k\circ p\circ i
\qquad \quad
\hbox{for all $k\in \Hom(g,f)$.}
$$
By (G2-I) there exists a pushout
$$
C\stackrel{r}{\rightarrow}D\stackrel{s}{\leftarrow}C
$$
of $C\stackrel{p\circ i}{\longleftarrow}I\stackrel{q\circ i}{\longrightarrow} C$ in $\D$.
 The universal property of $C\stackrel{r}{\rightarrow}D\stackrel{s}{\leftarrow}C$ implies that, for each $k\in \Hom(g,f)$, there exists a unique cover $\psi(k):D\to B$ such that the diagram

\begin{table}[H]
\centering
\begin{tikzpicture}[descr/.style={fill=white}]
\matrix(m)[matrix of math nodes,
row sep=2.6em, column sep=2.8em,
text height=1.5ex, text depth=0.25ex]
{
I
&C
&
\\
C
&D
&
\\
&
&B\\
};
\path[->,font=\scriptsize,>=angle 90]
(m-1-1) edge node[above] {$q\circ i$} (m-1-2)
        edge node[left]  {$p\circ i$} (m-2-1)
(m-1-2) edge node[right] {$s$} (m-2-2)
        edge[bend left] node[right] {$\phi(k)$} (m-3-3)
(m-2-1) edge node[below] {$r$} (m-2-2)
        edge[bend right] node[below] {$k$} (m-3-3);
 \path[->,dashed,font=\scriptsize,>=angle 90]
(m-2-2) edge node[descr] {$\psi(k)$} (m-3-3);
\end{tikzpicture}
\end{table}
\noindent commutes. The universal property of $C\stackrel{r}{\rightarrow}D\stackrel{s}{\leftarrow}C$ also implies the existence of a unique cover $h:D\to A$ such that the diagram

\begin{table}[H]
\centering
\begin{tikzpicture}[descr/.style={fill=white}]
\matrix(m)[matrix of math nodes,
row sep=2.6em, column sep=2.8em,
text height=1.5ex, text depth=0.25ex]
{
I
&C
&
\\
C
&D
&
\\
&
&A\\
};
\path[->,font=\scriptsize,>=angle 90]
(m-1-1) edge node[above] {$q\circ i$} (m-1-2)
        edge node[left]  {$p\circ i$} (m-2-1)
(m-1-2) edge node[right] {$s$} (m-2-2)
        edge[bend left] node[right] {$g$} (m-3-3)
(m-2-1) edge node[below] {$r$} (m-2-2)
        edge[bend right] node[below] {$g$} (m-3-3);
 \path[->,dashed,font=\scriptsize,>=angle 90]
(m-2-2) edge node[descr] {$h$} (m-3-3);
\end{tikzpicture}
\end{table}
\noindent commutes. The above diagram still commutes if $h$ is replaced by $f\circ \psi(k)$ for any $k\in \Hom(g,f)$.
By the uniqueness of $h$ we have
$$
h=f\circ \psi(k)
\qquad \quad
\hbox{for each $k\in \Hom(g,f)$}.
$$
In other words $\psi(k)\in \Hom(h,f)$ for all $k\in \Hom(g,f)$. Hence $\psi$ defines a map from $\Hom(g,f)$ to $ \Hom(h,f)$. To see the bijectivity of $\psi$, we consider the   map $r^*:\Hom(h,f)\to \Hom(g,f)$ induced from $r$. By the construction of $\psi$ we see that
$$
\psi \circ r^*=1.
$$
Given any $k\in \Hom(g,f)$, in view of
$\psi(k)\circ r=k$, the map $r^*:\Hom(h,f)\to \Hom(g,f)$ sends $\psi(k)$ to $k$. This shows that
$$
r^*\circ \psi=1.
$$
Therefore $r^*$ is the inverse of $\psi$. Hence $\psi$ is a bijection and  $|\Hom(h,f)|=\deg f$.

Since $g=h\circ r$, we have $h\sqsubseteq g$. The minimality of $g$ forces $g\sqsubseteq h$. Therefore $g$ is isomorphic to $h$. By Corollary~\ref{cor:iso_cover} the two covers $r,s\in \Hom(g,h)$ are isomorphisms . By the universal property of $C\stackrel{p}{\leftarrow}U\stackrel{q}{\rightarrow}C$ there exists a unique arrow $t:D\to U$ such that the diagram

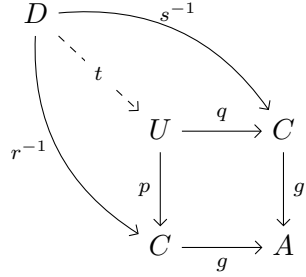
\begin{table}[H]
\centering
\begin{tikzpicture}[descr/.style={fill=white}]
\matrix(m)[matrix of math nodes,
row sep=2.6em, column sep=2.8em,
text height=1.5ex, text depth=0.25ex]
{
D
&
&
\\
&U
&C
\\
&C
&A\\
};
\path[->,font=\scriptsize,>=angle 90]
(m-2-2) edge node[above] {$q$} (m-2-3)
        edge node[left]  {$p$} (m-3-2)
(m-2-3) edge node[right] {$g$} (m-3-3)
(m-3-2) edge node[below] {$g$} (m-3-3)
(m-1-1) edge[bend left] node[above] {$s^{-1}$} (m-2-3)
        edge[bend right] node[left] {$r^{-1}$} (m-3-2);
 \path[->,dashed,font=\scriptsize,>=angle 90]
(m-1-1) edge node[descr] {$t$} (m-2-2);
\end{tikzpicture}
\end{table}
\noindent commutes. By (G3) there exists a unique $j\in \Sigma(U)$ such that $\Hom(t,j)\not=\emptyset$. By Proposition \ref{prop:iso}, $\deg r^{-1}=1$ and $\deg s^{-1}=1$ and this forces that $\deg (p\circ j)=1$ and $\deg (q\circ j)=1$ by (G4-I). If we can show $i=j$, then the theorem will follow.

By Proposition \ref{prop:iso} both $p\circ j$ and $q\circ j$ are isomorphisms. Clearly the cover $u=(p\circ j)^{-1}\circ p\circ i$ satisfies
$$
p\circ j\circ u=p\circ i.
$$
It follows from the construction of $r$ and $s$ that $r\circ p\circ i=s\circ q\circ i$. On the other hand,
by the choice of $j$ we have $(p\circ j)^{-1}\circ r^{-1}=(q\circ j)^{-1}\circ s^{-1}$.
These two relations combined yields
$$
q\circ j\circ u=q\circ i.
$$
Hence  we conclude $i=j$ from Lemma~\ref{lem:comp=}, as desired.
\end{proof}

\section{Two algorithms to find Galois closures}\label{sec:algo}

In this section we give two algorithms, Algorithm $\mathds I$ and Algorithm $\mathds R$, to find Galois closures of a given cover. They are stated in \S\ref{sec:algoIR}. Theoretically speaking, each step in Algorithms $\mathds I$ and $\mathds R$ can be executed under (G1)--(G4). More precisely, the conditions (G1), (G3) and (G4) guarantee the existence of the pullbacks of any two covers with the same codomain, the set $\Sigma(U)$ for any object $U$ of $\C$, and $\deg f$ for any cover $f$, respectively. In theoretical computer science, an algorithm is said to be {\it correct} (cf. \cite[p. 6]{MIT}) if it terminates with an output possessing the desired property. We verify in \S\ref{sec:algoI} the correctness
of Algorithm $\mathds I$ under the following necessary assumptions which ensure the termination of Algorithm $\mathds I$:
\begin{enumerate}
\item[(A1)] For any diagram $B\rightarrow A\leftarrow C$ in $\D$, its pullbacks in $\C$ can be identified in finite time.

\item[(A2)] For any object $U$ of $\C$, the set $\Sigma(U)$ can be determined in finite time.

\item[(A3)] For any arrow $f$ of $\D$, $\deg f$ can be evaluated in finite time.
\end{enumerate}
\noindent
In \S\ref{sec:algoR} we show that together with an additional condition (G5) the Algorithm $\mathds R$ is correct.

\subsection{Algorithm $\mathds I$ and Algorithm $\mathds R$}\label{sec:algoIR}

We now state the two algorithms to produce the Galois closures of a cover. The first algorithm is an iterative method:

\begin{center}
{\bf Algorithm $\mathds I$}
\end{center}

\medskip

Let $f:B\to A$ denote the input cover. Initially set $g:C\to A$ to be the cover $f$ and $n=1$. Begin with (1).
\begin{enumerate}
\item[(1)] If $n\geq \deg f-1$, then output $g$. Else go to (2).

\item[(2)] Find a pullback $B\stackrel{p}{\leftarrow }U\stackrel{q}{\rightarrow} C$ of $B\stackrel{f}{\rightarrow} A\stackrel{g}{\leftarrow} C$ and the set $\Sigma(U)$ of $U$. Go to (3).

\item[(3)] If $\deg (q\circ i)=1$ for all $i\in \Sigma(U)$, then output $g$. Else go to (4).

\item[(4)] Choose an $i\in \Sigma(U)$ with $\deg (q\circ i)>1$ and go to (5).

\item[(5)] Reset $g$ to be $g\circ q\circ i$ and increase $n$ by one. Go back to (1).
\end{enumerate}

\medskip

The second one is a divide-and-conquer solution for the Galois closure problem:

\medskip

\begin{center}
{\bf Algorithm $\mathds R$}
\end{center}

For an input cover $g$, denote by $\mathds R(g)$ the output cover after applying Algorithm $\mathds R$.
Given a cover $f:B\to A$, Algorithm $\mathds R$ begins with (1).
\medskip

\setlist[enumerate,1]{leftmargin=2em}

\begin{enumerate}
\item[(1)] If $\deg f\leq 2$, then output $f$. Else go to (2).

\item[(2)] Find a pullback $B\stackrel{p}{\leftarrow }U\stackrel{q}{\rightarrow} B$ of $B\stackrel{f}{\rightarrow} A\stackrel{f}{\leftarrow}B$ and go to (3).

\item[(3)]
Set $S=\{\mathds{R}(p\circ i)~|~i\in \Sigma(U)\}$ and go to (4).

\item[(4)]
While $|S|>1$ do the following:

\begin{enumerate}
\item[(i)] Pick any two distinct covers $g:C\to B$ and $h:D\to B$ from $S$ and go to (ii).

\item[(ii)] Find a pullback $C\stackrel{r}{\leftarrow }V\stackrel{s}{\rightarrow} D$ of $C\stackrel{g}{\rightarrow }B\stackrel{h}{\leftarrow} D$ and go to (iii).

\item[(iii)] Choose an $i\in \Sigma(V)$. Reset $S$ to be $S\cup\{g\circ r\circ i\}\setminus\{g,h\}$.
\end{enumerate}

\item[(5)] Output $f\circ g$ where $g$ is the sole cover in $S$.
\end{enumerate}

The correctness of Algorithm $\mathds R$ relies on one more hypothesis:

\setlist[enumerate,1]{leftmargin=2.7em}
\begin{enumerate}
\item[{\rm (G5)}] For each object $U$ of $\C$, each arrow in the set $\Sigma(U)$ is monic.
\end{enumerate}

\noindent It is worth pointing out that the categories in Examples~\ref{ex:field}--\ref{ex:variety} all satisfy (G5).

In both algorithms each step is practically doable under the assumptions (A1)--(A3).
In the next two subsections we show that the correctness of Algorithm $\mathds I$ under these assumptions followed by the correctness of Algorithm $\mathds R$ under the extra condition (G5). As a by-product, it is shown that the Galois closures of a given cover $f$ exist and are unique up to isomorphism.
Moreover, the degree of a Galois closure $g$ of $f$ is at most $(\deg f)!$, and if (G5) holds then $\deg g$ divides $(\deg f)!$.

For Example \ref{ex:field}, (A1) refers to the tensor product of two fields over a base field; (A2) refers to the factorization of polynomials over fields;
(A3) amounts to the evaluation of the degrees of irreducible polynomials over fields. Clearly (A3) is easily implemented. If $E$ is the field of rational numbers $\mathbb Q$  or a finite field, then (A1) and (A2) hold by \cite[Chapters 3 and 4]{Cohen:96} and \cite[Chapter 14]{Gathen:2013}, respectively. In these cases the Algorithms $\mathds I$ and $\mathds R$ are correct. The algorithm given in the Introduction agrees with Algorithm $\mathds I$. However, if $E$ is chosen to be an infinite extension of $\mathbb Q$, such as the one described in \cite[\S1]{Frohlich:55}, then (A2) does not hold and it is even unable to determine the irreducibility of polynomials over $E$  in finite time.

For Example~\ref{ex:graph}, (A1) refers to the fiber product of two unramified covering maps of a connected graph; (A2) refers to the identification of the connected components of a graph; (A3) means the evaluation of the degrees of unramified covering maps of connected graphs. If $G$ is a finite connected graph, then (A1) and (A3) are easily accomplished and (A2) holds by the depth-first search \cite[Chapter 22.5]{MIT}. In \cite[Example 17.5]{Terras:10} Terras described a way to construct a Galois closure of a given unramified covering map of finite graphs using Frobenius automorphisms. Both of our algorithms avoid the use of more sophiscated information like Frobenius automorphisms.

\subsection{Correctness of Algorithm $\mathds I$}\label{sec:algoI}

The aim of this subsection is to prove the correctness of Algorithm $\mathds I$ under (A1)--(A3). As a consequence of  Algorithm $\mathds I$, the Galois closures $g$ of any given cover $f$ exist and are unique up to isomorphism. Moreover we obtain the upper bound $\deg g\leq (\deg f)!$.

Our strategy is to use Theorem \ref{thm:suffgalois} to prove that the output of Algorithm $\mathds I$
is a Galois closure of a given cover $f$. For this, we need to control the sizes of various hom-sets after each iteration. This is done in the following four lemmas, with Lemma \ref{lem:loopinv1} and Lemma \ref{lem:loopinv2} giving the main conclusions and the lemmas preceding each of them providing the key inductive argument.

\begin{lem}\label{lem:increase}
Let $f:B\to A$ and $g:C\to A$ denote two covers. Let
$B\stackrel{p}{\leftarrow} U\stackrel{q}{\rightarrow} C$ be a pullback of $B\stackrel{f}{\rightarrow} A\stackrel{g}{\leftarrow} C$. If $i\in \Sigma(U)$ with $\deg (q\circ i)>1$ then
$$
|\Hom(g\circ q\circ i,f)|\geq |\Hom(g,f)|+1.
$$
\end{lem}
\begin{proof}
Consider the set
$$
S=\{j\in \Sigma(U)~|~\deg (q\circ j)=1\}.
$$
By Theorem~\ref{thm:homset} the set $\Hom(g,f)$ exactly consists of $p\circ j\circ (q\circ j)^{-1}$ for all $j\in S$. To prove the lemma, it suffices to show that the $n+1$ covers
$$
p\circ i
 \qquad
\hbox{and}
\qquad
p\circ j\circ (q\circ j)^{-1}\circ q\circ i
\qquad \quad \hbox{for all $j\in S$}
$$
in $\Hom(g\circ q\circ i,f)$ are distinct. Since $q\circ i$ is epic by (G2-II) the latter $n$ covers are pairwise distinct. Suppose that there exists an arrow $j\in S$ such that
$$
p\circ j\circ (q\circ j)^{-1}\circ q\circ i=p\circ i.
$$
Setting $u=(q\circ j)^{-1}\circ q\circ i$, we rewrite the above equality as
$
p\circ j\circ u=p\circ i.
$
On the other hand, the definition of $u$ implies
$
q\circ j\circ u=q\circ i.
$
Therefore $i=j$ by Lemma~\ref{lem:comp=}, which contradicts $\deg (q\circ i)>1$ and $\deg (q\circ j)=1$. Therefore the $n+1$ covers are distinct and the lemma follows.
\end{proof}

\begin{lem}\label{lem:loopinv1}
In {\rm Algorithm $\mathds I$} there holds the loop invariant:
\begin{gather*}\label{e:algoI}
|\Hom(g,f)|\geq n
\qquad
\hbox{at the start of the $n$th iteration.}
\end{gather*}
\end{lem}
\begin{proof}
Proceed by induction on $n$. Initially the cover $g$ is assigned to be $f$. Since $1_B\in \Hom(f,f)$ the estimate holds for $n=1$. In general, the estimate follows from the induction hypothesis and Lemma \ref{lem:increase}.
\end{proof}

\begin{lem}\label{lem:bound}
Let $f:B\to A$ and $g:C\to A$ denote two covers.
Let $B\stackrel{p}{\leftarrow}U\stackrel{q}{\rightarrow} C$ denote a pullback of $B\stackrel{f}{\rightarrow} A\stackrel{g}{\leftarrow} C$. If $h$ is a cover with $|\Hom(h,f)|=\deg f$ and $|\Hom(h,g)|=\deg g$, then
$$
|\Hom(h, g\circ q\circ i)|=\deg (g\circ q\circ i)
$$
for each $i\in \Sigma(U)$.
\end{lem}
\begin{proof}
Consider the set
$$
S=\{(i,k)~|~i\in \Sigma(U) ~\hbox{and}~k\in \Hom(h,g\circ q\circ i)\}.
$$
We estimate $|S|$ in two different ways. First, by construction we have
$$
|S|=\sum_{i\in \Sigma(U)} |\Hom(h,g\circ q\circ i)|.
$$
Second, since $f\circ p=g\circ q$ there is a map $\phi:S\to \Hom(h,f)\times \Hom(h,g)$ given by
$$
(i,k)
\quad \mapsto \quad
(p\circ i\circ k,q\circ i\circ k)
\qquad \quad
\hbox{for all $(i,k)\in S$}.
$$
We now show that $\phi$ is surjective. Write $h:D\to A$.
Let $r\in \Hom(h,f)$ and $s\in \Hom(h,g)$ be given. By the universal property of $B\stackrel{p}{\leftarrow}U\stackrel{q}{\rightarrow} C$ there exists a unique arrow $u:D\to U$ such that the diagram
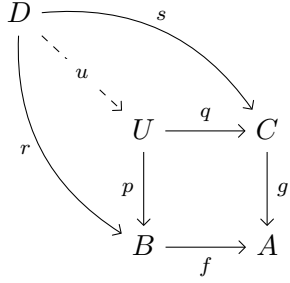
\begin{table}[H]
\centering
\begin{tikzpicture}[descr/.style={fill=white}]
\matrix(m)[matrix of math nodes,
row sep=2.6em, column sep=2.8em,
text height=1.5ex, text depth=0.25ex]
{
D
\\
&U
&C
\\
&B
&A
\\
};
\path[->,font=\scriptsize,>=angle 90]
(m-1-1) edge[bend left] node[above] {$s$} (m-2-3)
        edge[bend right] node[left] {$r$} (m-3-2)
(m-2-2) edge node[above] {$q$} (m-2-3)
        edge node[left] {$p$} (m-3-2)
(m-2-3) edge node[right] {$g$} (m-3-3)
(m-3-2) edge node[below] {$f$} (m-3-3);
\path[->,dashed,font=\scriptsize,>=angle 90]
(m-1-1) edge node[descr] {$u$} (m-2-2);
\end{tikzpicture}
\end{table}
\noindent commutes. By (G3) there exists a unique arrow $i\in \Sigma(U)$ with $\Hom(u,i)\not=\emptyset$. Pick any $k\in \Hom(u,i)$. Note that $k \in \Hom(h, g \circ q \circ i)$ and $\phi$ sends $(i,k)$ to $(r,s)$. This proves our claim.
Therefore
$|S|\geq |\Hom(h,f)|\times|\Hom(h,g)|$.
As $|\Hom(h, f)| = \deg f$ and $|\Hom(h, g)| = \deg g$ by hypotheses, we get
$$
\sum_{i\in \Sigma(U)} |\Hom(h,g\circ q\circ i)|=|S|\geq \deg f\cdot \deg g.
$$
Theorem~\ref{thm:homset} implies that
$|\Hom(h,g\circ q\circ i)|
\leq \deg (g\circ q\circ i)$ for all $i\in \Sigma(U)$.
Therefore
$$
\sum_{i\in \Sigma(U)}\deg (g\circ q\circ i)
\geq
\deg f\cdot \deg g.
$$
But the left-hand side of the above inequality is equal to $\deg f\cdot \deg g$ by (G4-I) and Lemma~\ref{lem:pullback}.
This shows that $|\Hom(h,g\circ q\circ i)|= \deg (g\circ q\circ i)$ for each $i\in \Sigma(U)$.
\end{proof}

\begin{lem}\label{lem:loopinv2}
In {\rm Algorithm $\mathds I$} the following loop invariant holds for any cover $h$ with $|\Hom(h,f)|=\deg f$:
$$
|\Hom(h,g)|=\deg g
\qquad
\hbox{at the start of each iteration}.
$$
\end{lem}
\begin{proof}
Proceed by induction on the number of iterations. Initially, the cover $g$ is assigned to be $f$. Therefore the equality holds for the first iteration. In general, the equality follows from the induction hypothesis and Lemma \ref{lem:bound}.
\end{proof}

We are now ready to prove

\setlist[enumerate,1]{leftmargin=2em}
\begin{thm}\label{thm:algoI}
\begin{enumerate}
\item {\rm Algorithm $\mathds I$} terminates if {\rm (A1)--(A3)} hold.

\item The output $g$ of {\rm Algorithm $\mathds I$} is a Galois closure of the input $f$.
\end{enumerate}
\noindent Consequently {\rm Algorithm $\mathds I$} is correct if {\rm (A1)--(A3)} hold.
\end{thm}
\begin{proof} (i) We begin with the termination of Algorithm $\mathds I$.
The steps (1), (3) and (4) can be finished in finite time by (A2) and (A3). Step (2) can be carried out by (A1) and (A2). It is clear that step (5) is fulfilled in finite time.
Hence it remains to indicate that the iterative loop in Algorithm $\mathds I$ halts. To see this, note that the initial value of $n$ is $1$ and $n$ is increased by one at step (5) in each iteration, the halting condition
$$
n\geq \deg f-1
$$
at step (1) would be satisfied in the $(\deg f-1)$th iteration if the halting condition at step (3) never holds. Therefore Algorithm $\mathds I$ terminates for all valid inputs.

(ii) We now invoke Theorem \ref{thm:suffgalois} to show that the output $g$ of Algorithm $\mathds I$ is a Galois closure of $f$. First we check that the output $g$ satisfies $|\Hom(g,f)|=\deg f$. This is so if the output is from (3), resulting from Theorem~\ref{thm:homset}.
Now suppose that $g$ is the output from (1). This happens when $n \ge \deg f-1$. Combined with Lemma~\ref{lem:loopinv1} we have
$$
|\Hom(g,f)|\geq \deg f-1.
$$
Denote by $B\stackrel{p}{\leftarrow }U\stackrel{q}{\rightarrow} C$ a pullback of $B\stackrel{f}{\rightarrow} A\stackrel{g}{\leftarrow} C$. By Theorem~\ref{thm:homset} there are at least $\deg f-1$ arrows $i\in \Sigma(U)$ with $\deg (q\circ i)=1$.  By Lemma~\ref{lem:pullback} this forces that there are exactly $\deg f$ arrows $i\in \Sigma(U)$ with $\deg (q\circ i)=1$. Therefore $|\Hom(g,f)|=\deg f$ by Theorem~\ref{thm:homset}. By Lemma~\ref{lem:loopinv2}, for any cover $h$ with $|\Hom(h,f)|=\deg f$ we have $|\Hom(h,g)|=\deg g$ and in particular $g\sqsubseteq h$. We have shown that $g$ is a least cover with respect to the property $|\Hom(g,f)|=\deg f$. By Theorem \ref{thm:suffgalois} the output $g$ is a Galois closure of $f$.
\end{proof}

An immediate consequence of Theorem \ref{thm:algoI} is

\begin{cor}\label{cor:deg<=2}
If $f$ is a cover with $\deg f\leq 2$, then $f$ is a Galois cover.
\end{cor}
\begin{proof}
For the input cover $f$ with $\deg f\leq 2$, Algorithm $\mathds I$ always outputs $f$. So $f$ is its own Galois closure by Theorem \ref{thm:algoI}(ii), hence it is a Galois cover.
\end{proof}

As mentioned in the beginning of \S\ref{sec:algo}, each step in Algorithm $\mathds I$ is theoretically doable under (G1)--(G4). Hence, we can draw the following conclusion from Theorem \ref{thm:algoI}.

\begin{thm}\label{thm:existgc} Let $f$ be a cover.
Then there exists a unique Galois closure of $f$ up to isomorphism.
Moreover the following statements are equivalent:
\begin{enumerate}
\item $g$ is a Galois closure of $f$.

\item $g$ is a least cover with respect to the property that $g$ is Galois and $f\sqsubseteq g$.

\item $g$ is a least cover with respect to the property $|\Hom(g,f)|=\deg f$.

\item $g$ is a minimal cover with respect to the property $|\Hom(g,f)|=\deg f$.

\end{enumerate}
\end{thm}
\begin{proof}
Theorem~\ref{thm:algoI}(ii) implies the existence of a Galois closure of $f$. The proof of Theorem~\ref{thm:algoI}(ii) actually shows that the output $g$ of Algorithm $\mathds I$ satisfies (iii), which proves the uniqueness of Galois closures of $f$ up to isomorphism.

(ii) $\Rightarrow$ (i): Immediate from Definition \ref{defn:galoisclosure}.

(i) $\Rightarrow$ (ii): Immediate from the uniqueness of Galois closures up to isomorphism.

(i) $\Rightarrow$ (iii): As mentioned above.

(iii) $\Rightarrow$ (iv): Obvious.

(iv) $\Rightarrow$ (i): Immediate from Theorem \ref{thm:suffgalois}.
\end{proof}

We end this subsection with an upper bound of the degree of a Galois closure.

\begin{lem}\label{lem:loopinv4}
In Algorithm $\mathds I$ the following loop invariant holds:
$$
\deg g \leq \prod_{k=1}^n
(\deg f-k+1)
\qquad
\hbox{at the start of the $n$th iteration}.
$$
\end{lem}
\begin{proof}
We show the loop invariant by induction on $n$.
 At the start of the first iteration the bound holds since $g=f$.
Now assume $n>1$. Consider the moment before executing (5) in the $(n-1)$th iteration.
 By Theorem~\ref{thm:homset}, $|\Hom(g, f)|$ is equal to the number of arrows $j\in \Sigma(U)$ such that $\deg (q \circ j) = 1$, and Lemma \ref{lem:loopinv1} says that this number is at least $n-1$. In step (4), an arrow $i\in \Sigma(U)$ with $\deg (q\circ i)>1$ is picked.
By Lemma~\ref{lem:pullback} it follows that
$$
\deg (q \circ i) \le \deg f - n+1.
$$
By (G4-I) and induction hypothesis, we get
$$
\deg (g \circ q \circ i) = \deg g \cdot \deg (q \circ i) \le (\deg f - n+1) \prod_{k=1}^{n-1}(\deg f-k+1) = \prod_{k=1}^{n}(\deg f-k+1).
$$
In step (5) the cover $g$ is replaced by $g\circ q\circ i$.  Hence the lemma holds.
\end{proof}

\begin{thm}\label{thm:lessthan}
If $g$ is a Galois closure of a cover $f$, then
$\deg g\leq (\deg f)!$.
\end{thm}
\begin{proof}
By Theorem \ref{thm:existgc}, we may assume that $g$ is an output of Algorithm $\mathds I$ with the input $f$. Then the theorem is immediate from Lemma \ref{lem:loopinv4}.
\end{proof}

\subsection{Correctness of Algorithm $\mathds R$}\label{sec:algoR}

In this subsection, we verify the correctness of Algorithm $\mathds R$ under (A1)--(A3) and (G5).

By Corollary \ref{cor:deg<=2} any cover $f$ with $\deg f \le 2$ is Galois, hence we are concerned with covers $f$ with $\deg f \ge 3$. The correctness of Algorithm $\mathds R$ is built inductively on the degree of the input covers. The step (3) in Algorithm $\mathds R$ for $f$ involves the outputs of covers $p \circ i$. The first thing is to verify that these covers do have degrees less than $\deg f$ so that the algorithm is valid.

\begin{lem}\label{lem:recursion}
Let the notation be as in the statement of {\rm Algorithm $\mathds R$}. If $\deg f>2$ then
$$
\deg (p\circ i)<\deg f
\qquad
\hbox{for each $i\in \Sigma(U)$}.
$$
\end{lem}
\begin{proof}
By Corollary \ref{cor:atleast1} there exists an arrow $j \in \Sigma(U)$ with $\deg (p \circ j)=1$. On the other hand, Lemma~\ref{lem:pullback} implies
$$
\sum_{i\in \Sigma(U)}\deg (p\circ i)=\deg f > 2,
$$
from which it follows that $|\Sigma(U)| \geq 2$ and hence $
\deg (p\circ i)<\deg f$ for all $i\in \Sigma(U)$.
\end{proof}

Similar to what we did for Algorithm $\mathds I$, in the next two lemmas, we control the sizes of various hom-sets occurring in Algorithm $\mathds R$.

\begin{lem}\label{lem:loopinv3}
Let the notation be as in the statement of {\rm Algorithm $\mathds R$}. Assume that $\mathds R(p\circ i)$ is a Galois closure of $p\circ i$ for each $i\in \Sigma(U)$. Then,
for any cover $k$ with $|\Hom(k,p\circ i)|=\deg (p\circ i)$ for each $i\in \Sigma(U)$, the while loop {\rm (4)} in {\rm Algorithm $\mathds R$} maintains the loop invariant:
\begin{gather*}
|\Hom(k,g)|=\deg g
\qquad
\hbox{for each $g\in S$}
\end{gather*}
at the start of each loop iteration.
\end{lem}
\begin{proof}
Proceed by induction on the number of loop iterations.
By Theorem \ref{thm:existgc}(iii) we have $\mathds R(p\circ i)\sqsubseteq k$ for all $i\in \Sigma(U)$.
Applying Corollary \ref{cor:char_galois} to $\mathds R(p\circ i)$ we have
$$
|\Hom(k,\mathds R(p\circ i))|=\deg \mathds R(p\circ i)
\qquad
\hbox{for all $i\in \Sigma(U)$}.
$$
Hence the lemma holds for the first loop.
In general the lemma follows from the induction hypothesis and Lemma \ref{lem:bound}.
\end{proof}

\begin{lem}\label{lem:indmap2}
Let $f:B\to A$ and $g:C\to A$ denote two covers. Let $B\stackrel{p}{\leftarrow}{U}\stackrel{q}{\rightarrow}C$ denote a pullback of $B\stackrel{f}{\rightarrow}{A}\stackrel{g}{\leftarrow}C$. If $k$ is a cover such that  $g\sqsubseteq f\circ k$, then
$$
|\Hom(f\circ k,g)|\leq |\Hom(k,p)|.
$$
\end{lem}
\begin{proof}
Let $D=\dom k$.
For each $\ell\in \Hom(f\circ k,g)$ the universal property of $B\stackrel{p}{\leftarrow}{U}\stackrel{q}{\rightarrow}C$ implies the existence of a unique arrow $\phi(\ell):D\to U$ such that the diagram
\begin{table}[H]
\centering
\begin{tikzpicture}[descr/.style={fill=white}]
\matrix(m)[matrix of math nodes,
row sep=2.6em, column sep=2.8em,
text height=1.5ex, text depth=0.25ex]
{
D
&
&
\\
&U
&C
\\
&B
&A\\
};
\path[->,font=\scriptsize,>=angle 90]
(m-2-2) edge node[above] {$q$} (m-2-3)
        edge node[left]  {$p$} (m-3-2)
(m-2-3) edge node[right] {$g$} (m-3-3)
(m-3-2) edge node[below] {$f$} (m-3-3)
(m-1-1) edge[bend left] node[above] {$\ell$} (m-2-3)
        edge[bend right] node[left] {$k$} (m-3-2);
 \path[->,dashed,font=\scriptsize,>=angle 90]
(m-1-1) edge node[descr] {$\phi(\ell)$} (m-2-2);
\end{tikzpicture}
\end{table}
\noindent commutes. This induces a map $\phi:\Hom(f\circ k,g)\to \Hom(k,p)$ that sends $\ell$ to $\phi(\ell)$ for all $\ell\in \Hom(f\circ k,g)$. Observe that the map $\phi$ satisfies the property:
$$
\ell=q\circ \phi(\ell)
\qquad
\hbox{for all $\ell\in \Hom(f\circ k,g)$}.
$$
The property implies that $\phi$ is injective:
For any $\ell,\ell'\in \Hom(f\circ k,g)$ with $\phi(\ell)=\phi(\ell')$ we have
$$
\ell=q\circ \phi(\ell)=q\circ \phi(\ell')=\ell'.
$$
By the injectivity of $\phi$ the lemma follows.
\end{proof}

\begin{thm}\label{thm:algoR}
\begin{enumerate}
\item {\rm Algorithm $\mathds R$} terminates if {\rm (A1)--(A3)} hold.

\item The output of {\rm Algorithm $\mathds R$} is a Galois closure of the input $f$ if {\rm (G5)} holds.
\end{enumerate}
\noindent Consequently {\rm Algorithm $\mathds R$} is correct if {\rm (A1)--(A3)} and {\rm (G5)} hold.
\end{thm}
\begin{proof} (i)
Step (1) can be implemented in finite time by (A3).
Step (2) asks for finding a pullback $B\stackrel{p}{\leftarrow} U\stackrel{q}{\rightarrow} B$ of $B\stackrel{f}{\rightarrow}A\stackrel{f}{\leftarrow}B$.
By (A1), $B\stackrel{p}{\leftarrow} U\stackrel{q}{\rightarrow} B$ can be identified in finite time.
In step (3) we have to construct the set $S=\{\mathds R(p\circ i)~|~i\in \Sigma(U)\}$.
By (A2) the set $\Sigma(U)$ can be determined in finite time.
In step (4) we do a while loop until $|S|=1$ and we return an output in step (5). Due to (A3) and step (4-iii) the while loop (4) terminates. Thus it remains to verify that  Algorithm $\mathds R$ terminates for the inputs $p\circ i$ with $i\in \Sigma(U)$.
By Lemma \ref{lem:recursion} we have $\deg (p\circ i)< \deg f$ for all $i\in \Sigma(U)$. By step (1), Algorithm $\mathds R$ terminates when the degree of the input cover is $\leq 2$. Therefore Algorithm $\mathds R$ eventually halts for the inputs $p\circ i$ with $i\in \Sigma(U)$. The termination of Algorithm $\mathds R$ follows.

(ii)
We show by induction on the degree of the input cover $f$ that $\mathds{R}(f)$ is a Galois closure of $f$. As remarked before, if $\deg f\leq 2$, then $\mathds R(f)=f$ is a Galois closure of $f$ by Corollary \ref{cor:deg<=2}. Suppose that $\deg f>2$.
Recall from step (5) that $\mathds R(f)=f\circ g$ where $g$ is the remaining cover in $S$.
We shall invoke Theorem \ref{thm:existgc}(iii) to prove that $f\circ g$ is a Galois closure of $f$. First we show that $|\Hom(f\circ g,f)|=\deg f$. As mentioned before, $\deg (p\circ i)<\deg f$ for all $i\in \Sigma(U)$. By induction hypothesis, for each $i\in \Sigma(U)$ the cover  $\mathds R(p\circ i)$ is a Galois closure of $p\circ i$. By Corollary \ref{cor:char_galois} we have
$$
|\Hom(\mathds R(p\circ i),p\circ i)|=\deg (p\circ i)
\qquad
\hbox{for all $i\in \Sigma(U)$}.
$$
Recall the construction of $g$ from the while loop (4), which implies that
$$
\mathds R(p\circ i)\sqsubseteq g
\qquad \hbox{for all $i\in \Sigma(U)$}.
$$
Applying Lemma \ref{lem:indmap} to $p\circ i\sqsubseteq \mathds R(p\circ i)\sqsubseteq g$ for all $i\in \Sigma(U)$, it follows that
$$
|\Hom(g,p\circ i)|=\deg (p\circ i)
\qquad \hbox{for all $i\in \Sigma(U)$}.
$$
Consider the set
$S=\{(i,k)~|~i\in \Sigma(U) ~\hbox{and}~k\in \Hom(g,p\circ i)\}$.
Applying the above equality we have
$$
|S|=\sum_{i\in \Sigma(U)} |\Hom(g,p\circ i)|
=\sum_{i\in \Sigma(U)} \deg (p\circ i).
$$
The latter sum is equal to $\deg f$ by Lemma~\ref{lem:pullback}. Therefore
$$
|S|=\deg f.
$$
By Theorem~\ref{thm:homset} we have $|\Hom(f\circ g,f)|\leq \deg f$. On the other hand, since $f\circ p=f\circ q$,  the $\deg f$ covers
$$
q\circ i\circ k
\qquad \quad
\hbox{for all $(i,k)\in S$}
$$
are in $\Hom(f\circ g,f)$. To assure $|\Hom(f\circ g,f)|=\deg f$, it remains to show that these covers are distinct. Suppose that $(i,k)$ and $(j,\ell)\in S$ satisfy $q\circ i\circ k = q\circ j\circ \ell$. Call $h$ the common cover.
By the universal property of $B\stackrel{p}{\leftarrow}U\stackrel{q}{\rightarrow}B$ there exists a unique arrow $u:C\to U$ such that the diagram

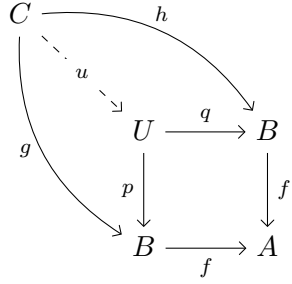
\begin{table}[H]
\centering
\begin{tikzpicture}[descr/.style={fill=white}]
\matrix(m)[matrix of math nodes,
row sep=2.6em, column sep=2.8em,
text height=1.5ex, text depth=0.25ex]
{
C
&
&
\\
&U
&B
\\
&B
&A\\
};
\path[->,font=\scriptsize,>=angle 90]
(m-2-2) edge node[above] {$q$} (m-2-3)
        edge node[left]  {$p$} (m-3-2)
(m-2-3) edge node[right] {$f$} (m-3-3)
(m-3-2) edge node[below] {$f$} (m-3-3)
(m-1-1) edge[bend left] node[above] {$h$} (m-2-3)
        edge[bend right] node[left] {$g$} (m-3-2);
 \path[->,dashed,font=\scriptsize,>=angle 90]
(m-1-1) edge node[descr] {$u$} (m-2-2);
\end{tikzpicture}
\end{table}
\noindent commutes. Since $k\in \Hom(u,i)$ and $\ell \in \Hom(u,j)$, it follows that $i=j$ by (G3) and then $k=\ell$ by (G5). Therefore $|\Hom(f\circ g,f)|=\deg f$, as claimed.

Next, let $h$ denote a cover with $|\Hom(h,f)|=\deg f$. We show that $f\circ g\sqsubseteq h$.
Fix a cover $k\in \Hom(h,f)$. Applying Lemma \ref{lem:indmap2} we have $|\Hom(h,f)|\leq |\Hom(k,p)|$.
Since $|\Hom(k,p)|\leq \deg p$ by Corollary~\ref{cor:homC} and $\deg p=\deg f$ by (G4-III), it forces that $|\Hom(k,p)|=\deg p$.
By the necessary condition of the equality $|\Hom(k,p)|=\deg p$ given in Corollary~\ref{cor:homC} we have
$$
|\Hom(k,p\circ i)|=\deg (p\circ i)
\qquad \quad
\hbox{for each $i\in \Sigma(U)$}.
$$
Now, we may apply Lemma \ref{lem:loopinv3} to conclude that $|\Hom(k,g)|=\deg g$. In particular $g\sqsubseteq k$ and hence $f\circ g\sqsubseteq f\circ k=h$, as desired. Therefore $f\circ g$ is a Galois closure of $f$ by Theorem \ref{thm:existgc}(iii).
\end{proof}

We end this section with a stronger bound for the degree of a Galois closure
than Theorem \ref{thm:lessthan} under the additional hypothesis (G5).

\begin{lem}\label{lem:all=}
Assume that $f:B\to A$ is a cover and $g:C\to A$ is a Galois cover. Let
$B\stackrel{p}{\leftarrow}U\stackrel{q}{\rightarrow}{C}$ denote a pullback of $B\stackrel{f}{\rightarrow}A\stackrel{g}{\leftarrow}{C}$.
Then
$$
|\Hom(p\circ i,p\circ j)|=\deg (p\circ j)
\qquad
\hbox{for all $i,j\in \Sigma(U)$}.
$$
In particular the following hold:
\begin{enumerate}
\item For all $i\in \Sigma(U)$ the covers $p\circ i$ are Galois.

\item For all $i\in \Sigma(U)$ the covers $p\circ i$ are isomorphic to each other.

\item $\deg (p\circ i)$ divides $\deg g$ for each $i\in \Sigma(U)$.
\end{enumerate}
\end{lem}
\begin{proof}
Fix $i:I\to U$ in $\Sigma(U)$.
Since $f\circ p=g\circ q$ it follows that $g\sqsubseteq f\circ p\circ i$. Applying Lemma \ref{lem:indmap2} with $k=p\circ i$ we have
$$
|\Hom(f\circ p\circ i,g)|\leq |\Hom(p\circ i,p)|.
$$
Since $g$ is Galois and by Corollary~\ref{cor:char_galois}, we have
$$
|\Hom(f\circ p\circ i,g)|=\deg g.
$$
By Corollary~\ref{cor:homC} we have $|\Hom(p\circ i,p)|\leq \deg p=\deg g$, where the equality follows from (G4-III). Concluding from the above estimations we obtain that
$$
|\Hom(p\circ i,p)|=\deg p.
$$
Therefore $|\Hom(p\circ i,p\circ j)|=\deg (p\circ j)$ for each $j\in \Sigma(U)$ by Corollary~\ref{cor:homC}. In particular we have $p\circ i\sqsubseteq p\circ j$.

Since $i$ is an arbitrary arrow in $\Sigma(U)$ the statement (ii) holds and (i) is immediate from Corollary \ref{cor:aut}. As a consequence of (ii) and Corollary \ref{cor:iso_cover} we have $\deg (p\circ i)=\deg (p\circ j)$ for all $i,j\in \Sigma(U)$. Combined with Lemma~\ref{lem:pullback} we have
$$
\deg g=|\Sigma(U)|\cdot \deg(p\circ i)
\qquad \quad
\hbox{for any $i\in \Sigma(U)$}.
$$
Therefore (iii) follows.
\end{proof}

\begin{thm}\label{thm:divid}
Under {\rm(G5)}, the degree of a Galois closure of a cover $f$ divides $(\deg f)!$.
\end{thm}
\begin{proof}
We prove by induction on $\deg f$ using Algorithm $\mathds R$. It is true if $\deg f\leq 2$ by Corollary \ref{cor:deg<=2}. Now suppose $\deg f>2$. Let $\Sigma(U)=\{i_1,i_2,\ldots,i_m\}$. According to step (3) of Algorithm $\mathds R$, the set $S$ initially consists of distinct Galois covers $\mathds R(p \circ i_1),\ldots,\mathds R(p \circ i_m)$ with $i_1,\ldots,i_m\in \Sigma(U)$. When $m>1$, in the while loop (4) we are to successively replace two covers in $S$ by a new cover and repeat this $m-1$ times until only one cover remains. We can do this in the following way. Define $g_0 = \mathds R(p \circ i_1)$. In the $j$th loop for $1 \le j \le m-1$, at (i) we choose the two covers $g_{j-1}$ and $\mathds R(p \circ i_{j+1})$ and at (iii) replace them by a new cover denoted by $g_j$. By Lemma \ref{lem:all=}(iii), in $j$th loop for all $1\leq j\leq m-1$,
we have
$\deg g_j$ divides
$\deg g_{j-1} \cdot \deg \mathds R(p \circ i_{j+1})$.
A routine induction yields that $\deg g_m$ divides
$$
\prod_{i\in \Sigma(U)} \deg \mathds R(p\circ i).
$$
By (G4-I) the degree of $\mathds R(f)=f\circ g_m$ is equal to $\deg f\cdot \deg g_m$ which divides
$$
\deg f\cdot \prod_{i\in \Sigma(U)} \deg (p\circ i)!
$$
by induction hypothesis.
It remains to prove that the second multiplicand divides $(\deg f-1)!$. Recall from Corollary \ref{cor:atleast1} that $\Sigma(U)$ contains at least one arrow $i$ such that $\deg (p\circ i)=1$. Hence
$$
\sum_{i\in \Sigma(U)\atop
\deg (p\circ i)>1} \deg (p\circ i)\leq \deg f-1
$$
by Lemma~\ref{lem:pullback}. It follows that
$$
\prod_{i\in \Sigma(U)}\deg (p\circ i)!
=
\prod_{i\in \Sigma(U)\atop
\deg (p\circ i)>1} \deg (p\circ i)!
$$
is a factor of $(\deg f-1)!$. This completes the proof of the theorem.
\end{proof}

\section{Remarks on ramified covers}

The concrete examples, Examples \ref{ex:graph}--\ref{ex:variety}, to which our categorical approach apply all concern finite unramified covers. In arithmetic geometry finite ramified covers have been vastly studied in the literature, see, for example, \cite{FI2002} and the references therein; they are closely tied to the celebrated Inverse Galois Problem. It is natural to ask whether there exists a unified categorical approach to finding the Galois closure of a finite ramified cover. 

We point out that, given a finite cover $f : B \rightarrow A$ of two irreducible normal projective varieties defined over a field $k$, with or without ramifications, the two algorithms given in \S4 can be applied to find a cover $g : C \rightarrow A$ of normal irreducible projective varieties defined over $k$ such that $g$ is a Galois closure of $f$. To see this,  it suffices to construct categories $\C$ and $\D$ both containing the arrow $f$ such that  properties (G1)-(G5) hold. The objects in $\C$ are normal projective varieties $V$ defined over $k$ which are finite covers of $A$. Denote by $k(V)$ the algebra of $k$-rational functions on $V$. An arrow $r:V \rightarrow W$ between two objects $V$ and $W$ in $\C$ is a finite morphism over $k$ which induces a nontrivial $k$-algebra homomorphism $r^*:k(W)\to k(V)$. In particular, when $V$ and $W$ are irreducible, $k(V)$ and $k(W)$ are fields and $r^*$ embeds $k(W)$ in $k(V)$, allowing $k(W)$ to be regarded as a subfield of $k(V)$. This will be used below repeatedly. The category $\D$, whose objects are those in $\C$ which are irreducible, is a full subcategory of $\C$. To verify (G1)--(G5), the argument is similar to the case of \'etale covers discussed in Example 2.4 except we have to make sure that the varieties are all normal.

{We start with (G1). Given $ Y \stackrel{r}{\rightarrow} X \stackrel{s}{\leftarrow} Z$ in $\D$, a pullback is described as follows. 
Decompose the tensor product $k(Y)\otimes_{k(X)} k(Z)$ into the direct sum of (finitely many) fields $K_j$, each being a finite extension of $k(Y)$ and $k(Z)$. Since $X$, $Y$ and $Z$ are normal by assumption,  $r:Y\to X$ and $s:Z\to X$ are the normalizations of $X$ in $k(Y)$ and $k(Z)$, respectively. For each $j$ the normalization $t_j:V_j\to X$ of $X$ in $K_j$ factors through $r$ and $s$ via the normalizations $p_j:V_j\to Y$  and $q_j:V_j\to Z$ of $Y$ and $Z$ in $K_j$, respectively. A pullback of  $ Y \stackrel{r}{\rightarrow} X \stackrel{s}{\leftarrow} Z$ is
$$
Y\stackrel{p}{\leftarrow} V\stackrel{q}{\rightarrow} Z
$$ 
where $V$ is the normal variety over $k$ whose irreducible components are the $V_j$'s and $p$ and $q$ are the covers such that $p|_{V_j}=p_j$ and $q|_{V_j}=q_j$ for each $j$.} Equivalently, this amounts to saying that we take the geometric pullback of $Y \times_X Z$ and replace each irreducible component $V_j'$ by its normalization $V_j$ in its function field so that we obtain an object in $\C$.

For (G2) it remains to describe a pushout of the given diagram $Y \stackrel{r}{\leftarrow} X \stackrel{s}{\rightarrow} Z$ in $\D$. As explained above, $k(Y)$ and $k(Z)$ are embedded in $k(X)$ as subfields via $r^*$ and $s^*$, respectively; denote by $F$ the intersection  $k(Y) \cap k(Z)$. The normal variety $X$ is covered by finitely many affine Spec$A_i$, with $A_i$ integrally closed in $k(X)$, which is also the quotient field of $A_i$. Then each $F_i = A_i \cap F$ is integrally closed in $F$, and by gluing Spec$F_i$ together we obtain a normal variety $V$ over $k$ whose function field is $F$. A pushout of $Y \stackrel{r}{\leftarrow} X \stackrel{s}{\rightarrow} Z$ is 
$$
Y \stackrel{p}{\rightarrow} V \stackrel{q}{\leftarrow} Z
$$ 
where $p$ and $q$ are the normalizations of $V$ in $k(Y)$ and $k(Z)$, respectively. Given an object $U$ of $\C$, each of its irreducible component $U_i$ is an object in $\D$, the set $\Sigma(U)$ is defined the same way as before so that (G3) holds. The degree function of an arrow $f : X \rightarrow Y$ in $\D$ is the degree of the function field extension $[k(X) : k(Y)]$ for irreducible normal varieties $X$ and $Y$.
It clearly satisfies (G4). Finally condition (G5) obviously is valid.

 Grothendieck defined a cover of varieties to be a finite flat morphism. Note that taking fiber product followed by normalization may not preserve flatness. A normal curve is automatically smooth, hence the finite morphisms between normal curves are flat. Thus if $f : B \rightarrow A$ is a finite cover between smooth curves, our algorithms output a smooth curve $C$ over $k$ so that the cover $g: C \rightarrow A$ is a Galois closure of $f$ in the sense of Grothendieck. On the other hand, if $A$ has dimension at least $2$, so do all objects in the category $\C$. The singular locus contained in an object of $\C$ is known to have codimension at least $2$.  The output cover $g : C \rightarrow A$ from our algorithms is flat outside the singular loci in $C$ and $A$. It is in this sense that $g$ is a Galois closure of $f$. If the cover $f: B \rightarrow A$ to begin with is flat between smooth varieties, there is no canonical way to choose a smooth variety $C'$ with the same function field as $C$ which gives rise to a Galois closure $g' : C' \rightarrow A$ of $f$ in the sense of Grothendieck.

A more sophiscated type of cover of varieties together with a permutation representation is discussed in \cite[\S 3.1.3]{fried:1999}, where a geometric method to find its Galois closure using fiber product is outlined. It would be interesting to see whether this approach can be described in terms of categorical language as well.

\bigskip

{\bf Acknowledgment.} The authors would like to thank Michael Fried for bringing to their attention that categorical approach to finding Galois closures for finite (not necessarily \'etale) covers of normal varieties has been used by people working on Galois covers and moduli spaces to obtain deep results in arithmetic geometry. The first papers in this field are \cite{Fried:91, Fried:92}, where the categorical construction was applied  to relate two different versions of Hurwitz spaces. The authors would like to express their deep gratitude to the anonymous referee for  inspirational discussions, clarifications, and the encouragement to include covers with ramifications.

\bibliographystyle{alpha}
\bibliography{GC}

\end{document}